\documentclass[11pt,final]{newarticle}
\usepackage{mathrsfs}
\usepackage{amssymb}
\usepackage{amsmath}
\usepackage{mathtools}
\usepackage[amsmath,thmmarks]{ntheorem}

\usepackage{xcolor}
\definecolor{darkblue}{RGB}{2,58,141}
\definecolor{lightgray}{RGB}{229,235,244}
\definecolor{lightviolet}{RGB}{196,150,241}

\usepackage{enumerate}
\usepackage{enumitem}
\setlist[enumerate,1]{label=(\arabic*).,ref=\thetheorem\,(\arabic*),
font=\textup,leftmargin=7mm,labelsep=1.5mm,topsep=0mm,itemsep=-0.8mm}
\setlist[enumerate,2]{label=(\alph*).,font=\textup,
leftmargin=7mm,labelsep=1.5mm,topsep=-0.8mm,itemsep=-0.8mm}

\usepackage{tcolorbox}
\tcbset{on line,arc=0pt,outer arc=0pt,colback=lightgray,colframe=darkblue,
width=\linewidth+3pt,boxsep=0pt,top=5pt,left=1.5pt,right=1.5pt,
boxrule=0pt,bottomrule=0.5pt,toprule=0.5pt}

\usepackage{geometry}
\usepackage[pdfstartview=FitH,bookmarksopen=false,
colorlinks=true,linkcolor=blue,citecolor=blue]{hyperref}

\usepackage{microtype}

\numberwithin{equation}{section}

\theoremstyle{plain}
\newtheorem{theorem}{Theorem}[section]
\newtheorem{lemma}{Lemma}[section]
\newtheorem{corollary}{Corollary}[section]

\theorembodyfont{\normalfont}
\newtheorem{definition}{Definition}[section]
\newtheorem{remark}{Remark}[section]

\theoremstyle{nonumberplain}
\theoremheaderfont{\it}
\theorembodyfont{\normalfont}
\theoremsymbol{\mbox{$\Box$}}
\newtheorem{proof}{Proof.}

\allowdisplaybreaks
\renewcommand{\cup}{\operatorname*{\scalebox{0.93}{$\bigcup$}}\limits}

\makeatletter
\renewcommand*\env@matrix[1][\arraystretch]{%
\edef\arraystretch{#1}%
\hskip -\arraycolsep
\let\@ifnextchar\new@ifnextchar
\array{*\c@MaxMatrixCols c}}
\makeatother

\definecolor{primaryL}{HTML}{B2DFDB}
\begin{document}
	
\begin{frontmatter}
\title{On the principal eigenvectors of uniform hypergraphs\tnoteref{titlenote}}
\tnotetext[titlenote]{This work was supported by the National Nature 
Science Foundation of China (Nos.\,11471210, 11101263)}

\author[]{Lele Liu}

\author[]{Liying Kang}

\author[]{Xiying Yuan\corref{correspondingauthor}}
\cortext[correspondingauthor]{Corresponding author}
\ead{xiyingyuan2007@hotmail.com}

\address{Department of Mathematics, Shanghai University, Shanghai 200444, China}

\begin{abstract}
Let $\mathcal{A}(H)$ be the adjacency tensor of $r$-uniform hypergraph $H$. If $H$ is 
connected, the unique positive eigenvector $x=(x_1,x_2,\cdots,x_n)^{\mathrm{T}}$ with 
$||x||_r=1$ corresponding to spectral radius $\rho(H)$ is called the principal eigenvector 
of $H$. The maximum and minimum entries of $x$ are denoted by $x_{\max}$ and $x_{\min}$, 
respectively. In this paper, we investigate the bounds of $x_{\max}$ and $x_{\min}$ in 
the principal eigenvector of $H$. Meanwhile, we also obtain some bounds of the ratio 
$x_i/x_j$ for $i$, $j\in [n]$ as well as the principal ratio $\gamma(H)=x_{\max}/x_{\min}$ 
of $H$. As an application of these results we finally give an estimate of the gap of spectral 
radii between $H$ and its proper sub-hypergraph $H'$.
\end{abstract}

\end{frontmatter}

\section{Introduction}

Let $G$ be a simple connected graph, and $A(G)$ be the adjacency matrix of $G$. 
Perron-Frobenius theorem implies that $A(G)$ has a unique unit positive eigenvector 
corresponding to spectral radius $\rho(G)$, which is usually called the principal eigenvector 
of $G$. The principal eigenvector plays an important role in spectral graph theory, and 
there exist some literatures concerning that. In 2000, Papendieck and Recht 
\cite{Papendieck:Maximal entries principal eigenvector} posed an upper bound for the 
maximum entry of the principal eigenvector of a graph. Later, Zhao and Hong \cite{ZhaoHong} 
further investigated bounds for maximum entry of the principal eigenvector of a symmetric 
nonnegative irreducible matrix with zero trace. In 2007, Cioab$\check{\text{a}}$ and 
Gregory \cite{Sebastian:Principal eigenvectors} improved the bound of Papendieck and 
Recht \cite{Papendieck:Maximal entries principal eigenvector} in terms of the vertex 
degree. Das \cite{Das:Sharp upper bound,Das:Maximal and minimal entry} obtained bounds 
for the maximum entry of the principal eigenvector of the signless Laplacian matrix and 
distance matrix in 2009 and 2011, respectively. Recently, Das et al. \cite{Das:Distance signless Laplacian} 
determine upper and lower bounds of the maximum entry of the principal eigenvector of 
the distance signless Laplacian matrix.  

Hypergraph is a natural generalization of ordinary graph (see \cite{Berge:Hypergraph}). 
A {\em hypergraph} $H=(V,E)$ on $n$ vertices is a set of vertices, say $V=\{1,2,\ldots,n\}$ 
and a set of edges, say $E=\{e_{1},e_{2},\ldots,e_{m}\}$, where $e_{i}=\{i_{1},i_{2},\ldots,i_{\ell}\}$, 
$i_j\in [n]:=\{1,2,\ldots,n\}$, $j\in[\ell]$. A hypergraph is called {\em $r$-uniform} if every 
edge contains precisely $r$ vertices. Let $H=(V,E)$ be an $r$-uniform hypergraph on $n$ vertices. 
The adjacency tensor (see \cite{Cooper:Spectra Uniform Hypergraphs}) of $H$ is defined as 
the order $r$ dimension $n$ tensor $\mathcal{A}(H)$ whose $(i_1i_2\cdots i_r)$-entry is
\[
(\mathcal{A}(H))_{i_1i_2\cdots i_r}=
\begin{dcases}
\frac{1}{(r-1)!} & \text{if}~\{i_1,i_2,\ldots,i_r\}\in E(H),\\
0 & \text{otherwise}.
\end{dcases}
\]

Qi \cite{Qi2005} and Lim \cite{Lim} independently introduced the concept of eigenvalues of tensors, 
from which one can get the definition of the eigenvalues of the adjacency tensor of an $r$-uniform 
hypergraph (see more in Section \ref{sec2}). Obviously, adjacency tensor is a symmetric nonnegative
tensor. The Perron-Frobenius theorem for nonnegative tensors has been established 
(see \cite{K.C.Chang.etc:Perron-Frobenius Theorem}, \cite{Yang:Nonegative Weakly Irreducible Tensors}
and the references in them). Based on the Perron-Frobenius theorem there exists a unique positive 
eigenvector $x$ with $||x||_r=1$ for $\mathcal{A}(H)$ of a connected $r$-uniform hypergraph $H$. 
This vector will be called the {\em principal eigenvector} of $H$. Denote by $x_{\max}$ the 
maximum entry of $x$ and by $x_{\min}$ the minimum entry of $x$. The {\em principal ratio}, 
$\gamma(H)$, of $H$ is defined as $x_{\max}/x_{\min}$.

Recently, Nikiforov \cite{Nikiforov-3} presented some bounds on the entry of the principal 
eigenvector of an $r$-uniform hypergraph (see Section 7 of \cite{Nikiforov-3}). Li et al. 
\cite{Li:Principal eigenvector} posed some lower bounds for principal eigenvector 
of connected uniform hypergraphs in terms of vertex degrees and the number of vertices. 
In this paper, we generalize some classical bounds on the principal eigenvector to an $r$-uniform 
hypergraph. In Section \ref{sec2}, we introduce some notations and necessary lemmas. In 
Section \ref{sec3}, we present some upper bounds on the maximum and minimum entries in the 
principal eigenvector $x=(x_1,x_2,\ldots,x_n)^{\mathrm{T}}$ of a connected $r$-uniform 
hypergraph $H$. Meanwhile, we also investigate bounds of the ratio $x_i/x_j$ for $i$, 
$j\in [n]$ as well as the principal ratio $\gamma(H)=x_{\max}/x_{\min}$ of $H$. Based on 
these results, in Section \ref{sec4} we finally give an estimate of the gap of spectral 
radii between $H$ and its proper sub-hypergraph $H'$.

\section{Preliminaries}
\label{sec2}

Let $H=(V(H),E(H))$ and $H'=(V(H'),E(H'))$ be two $r$-uniform hypergraphs. If $V(H')\subseteq V(H)$ 
and $E(H')\subseteq E(H)$, then $H'$ is called a {\em sub-hypergraph} of $H$. If $H'$ is a sub-hypergraph 
of $H$, and $H'\neq H$, then $H'$ is called a {\em proper sub-hypergraph} of $H$. A hypergraph $H$ is 
called a {\em linear} hypergraph provided that each pair of the edges of $H$ has at most one common 
vertex (see \cite{Bretto}). 
 
For a vertex $v\in V(H)$, the {\em degree} $d_{H}(v)$ is defined as the number of edges containing $v$.
We denote by $\delta(H)$ and $\Delta(H)$ the minimum and maximum degrees of the vertices of $H$.  
In a hypergraph $H$, two vertices $u$ and $v$ are {\em adjacent} if there is an edge $e$ of $H$ such 
that $\{u,v\}\subseteq e$. A vertex $v$ is said to be {\em incident} to an edge $e$ if $v\in e$. 
A {\em walk} of hypergraph $H$ is defined to be an alternating sequence of vertices and edges 
$v_1e_1v_2e_2\cdots v_{\ell}e_{\ell}v_{\ell+1}$ satisfying that both $v_{i}$ and $v_{i+1}$ are 
incident to $e_i$ for $1\leqslant i\leqslant\ell$. A walk is called a {\em path} if all vertices 
and edges in the walk are distinct. The {\em length} of a path is the number of edges in it. A 
hypergraph $H$ is called {\em connected} if for any vertices $u$, $v$, there is a path connecting 
$u$ and $v$. The {\em distance} between two vertices is the length of the shortest path connecting 
them. The {\em diameter} of a connected $r$-uniform hypergraph $H$ is the maximum distance among all 
vertices of $H$.

The following definition was introduced by Qi \cite{Qi2005}.

\begin{definition}
[\cite{Qi2005}]
Let $\mathcal{A}$ be an order $r$ dimension $n$ tensor, $x=(x_{1},x_2,\ldots,x_{n})^{\mathrm{T}}\in\mathbb{C}^{n}$ 
be a column vector of dimension $n$. Then $\mathcal{A}x^{r-1}$ is defined to be a vector in $\mathbb{C}^{n}$ 
whose $i$-th component is the following
\begin{equation}
\label{eq:eigenvector}
(\mathcal{A}x^{r-1})_i=\sum_{i_2,\cdots,i_r=1}^na_{ii_2\cdots i_r}x_{i_2}\cdots x_{i_r},~~
i=1,2,\ldots,n.
\end{equation}
If there exists a number $\lambda\in\mathbb{C}$ and a nonzero vector $x\in\mathbb{C}^{n}$ such that
\begin{equation}
\label{Eigenequations}
\mathcal{A}x^{r-1}=\lambda x^{[r-1]}, 
\end{equation}
then $\lambda$ is called an {\em eigenvalue} of $\mathcal{A}$, $x$ is called an {\em eigenvector} of 
$\mathcal{A}$ corresponding to the eigenvalue $\lambda$, where $x^{[r-1]}=(x_{1}^{r-1},x_{2}^{r-1},\ldots,x_{n}^{r-1})^{\mathrm{T}}$. 
The {\em spectral radius} of $\mathcal{A}$ is the maximum modulus of the eigenvalues of $\mathcal{A}$, i.e., 
$\rho(\mathcal{A})=\max\{|\lambda|:\lambda~\text{is an eigenvalue of}~\mathcal{A}\}$.
\end{definition}

For an $r$-uniform hypergraph $H$, the spectrum, eigenvalues and spectral radius of $H$ are defined to 
be those of its adjacency tensor $\mathcal{A}(H)$. By using the general product of tensors defined by 
Shao in \cite{Shao2013}, $\mathcal{A}x^{r-1}$ can be simply written as $\mathcal{A}x$. In the remaining 
part of this paper, we will use $\mathcal{A}x$ to denote $\mathcal{A}x^{r-1}$.

\begin{theorem}
[\cite{Qi2013}]
\label{relaigh}
Let $\mathcal{A}$ be a nonnegative symmetric tensor of order $r$ and dimension $n$,
denote $\mathbb{R}_{+}^{n}=\{x\in\mathbb{R}^{n}\,|\, x\geqslant 0\}$. Then we have
\begin{equation}
\label{Expression for spectral radius}
\rho(\mathcal{A})=\max\left\{x^{\mathrm{T}}(\mathcal{A}x)\,|\, x\in\mathbb{R}_{+}^{n},||x||_r=1\right\}.
\end{equation}
Furthermore, $x\in\mathbb{R}_{+}^{n}$ with $||x||_r=1$ is an optimal solution of the above 
optimization problem if and only if it is an eigenvector of $\mathcal{A}$ corresponding to 
the eigenvalue $\rho(\mathcal{A})$.
\end{theorem}

A novel method, weighted incidence matrix method is introduced by Lu and Man for computing the
spectral radii of hypergraphs in \cite{LuAndMan:Small Spectral Radius}. 

\begin{definition}
[\cite{LuAndMan:Small Spectral Radius}]
\label{defn:WeightedIncidenceMatrix}
A {\em weighted incidence matrix} $B$ of a hypergraph $H=(V(H),E(H))$ is a $|V(H)|\times |E(H)|$ 
matrix such that for any vertex $v$ and any edge $e$, the entry $B(v,e)>0$ if $v\in e$ and $B(v,e)=0$ 
if $v\notin e$.
\end{definition}

\begin{definition}
[\cite{LuAndMan:Small Spectral Radius}]
\label{defn:alpha-normal}
A hypergraph $H$ is called {\em $\alpha$-normal} if there exists a weighted incidence matrix $B$ 
satisfying\\
(1). $\sum_{e:v\in e}B(v,e)=1$, for any $v\in V(H)$; \\
(2). $\prod_{v:v\in e}B(v,e)=\alpha$, for any $e\in E(H)$. \\
Moreover, the weighted incidence matrix $B$ is called {\em consistent} if for any cycle 
$v_0e_1v_1\cdots e_{\ell}v_{\ell}$ $(v_{\ell}=v_{0})$
\[
\prod^{\ell}_{i=1}\frac{B(v_{i},e_{i})}{B(v_{i-1},e_{i})}=1.
\]
In this case, $H$ is called {\em consistently $\alpha$-normal}.
\end{definition}

\begin{lemma}
[\cite{LuAndMan:Small Spectral Radius}]
\label{lem:StrictlySubnormal}
Let $H$ be a connected $r$-uniform hypergraph. Then $H$ is consistently $\alpha$-normal 
if and only if $\rho(H)=\alpha^{-\frac{1}{r}}$.
\end{lemma}

\section{The principal eigenvectors of uniform hypergraphs}
\label{sec3}
\subsection{The extreme components of principal eigenvector}

In this subsection we shall present some bounds on $x_{\max}$ and $x_{\min}$ in the principal 
eigenvector of a connected $r$-uniform linear hypergraph $H$. 

The following lemma is very useful for us, so we reproduce the proof of Lu and Man
\cite{LuAndMan:Small Spectral Radius}.

\begin{lemma}
[\cite{LuAndMan:Small Spectral Radius}]
Suppose that $H$ is a connected $r$-uniform hypergraph. Let $H$ be consistently $\alpha$-normal 
with weighted incidence matrix $B$. If $x=(x_1,x_2,\ldots,x_n)^{\mathrm{T}}$ is the principal 
eigenvector of $H$, then for any edge 
$e=\{v_{i_1},v_{i_2},\ldots,v_{i_r}\}\in E(H)$, we have
\begin{equation}
\label{eq:B(v,e) is constant}
B(v_{i_1},e)^{\frac{1}{r}}\cdot x_{v_{i_1}}=
B(v_{i_2},e)^{\frac{1}{r}}\cdot x_{v_{i_2}}=
\cdots=
B(v_{i_r},e)^{\frac{1}{r}}\cdot x_{v_{i_r}}.
\end{equation}
\end{lemma}

\begin{proof}
Let $\rho$ be the spectral radius of $H$. From Lemma \ref{lem:StrictlySubnormal}, $H$ is 
consistently $\alpha$-normal with $\alpha=\rho^{-r}$. According to Theorem \ref{relaigh} 
and Definition \ref{defn:alpha-normal}, we have
\begin{align}
\label{eq:equality hold}
\rho=x^{\mathrm{T}}\left(\mathcal{A}(H)x\right) & 
=r\sum_{e=\{v_{i_1},v_{i_2},\cdots,v_{i_r}\}\in E(H)}x_{v_{i_1}}x_{v_{i_2}}\cdots x_{v_{i_r}} \notag\\
& =\frac{r}{\alpha^{\frac{1}{r}}}\sum_{e\in E(H)}\prod_{v_i:v_i\in e}
\left(B(v_i,e)^{\frac{1}{r}}x_{v_i}\right) \notag\\
& \leqslant\frac{r}{\alpha^{\frac{1}{r}}}\sum_{e\in E(H)}
\frac{\sum_{v_i:v_i\in e}B(v_i,e)x_{v_i}^r}{r} \\
& =\alpha^{-\frac{1}{r}}=\rho, \notag
\end{align}
which implies that in \eqref{eq:equality hold} equality holds. Therefore the desired 
equations \eqref{eq:B(v,e) is constant} follows. 
\end{proof}

\begin{theorem}
\label{thm:x_i upper bounds}
Let $H$ be a connected $r$-uniform linear hypergraph on $n$ vertices with spectral radius $\rho$ and 
principal eigenvector $x$. For any $v\in V(H)$, if $d$ is the degree of vertex $v$, then
\[
x_v\leqslant
\frac{1}{\sqrt[\leftroot{-2}\uproot{8}r]{1+(r-1)\left(\frac{\rho^r}{d}\right)^{\frac{1}{r-1}}}}.
\]
\end{theorem}

\begin{proof}
Let $e_1$, $e_2$, $\ldots$, $e_d$ be all distinct edges containing $v$. Denote
\[
e_i=\{v,v_{i_1},v_{i_2},\ldots,v_{i_{r-1}}\},~~i=1,2,\ldots,d.
\]
Suppose that $H$ is consistently $\alpha$-normal with weighted incidence matrix $B$, 
where $\alpha=\rho^{-r}$. For each edge $e_i$, by \eqref{eq:B(v,e) is constant} and 
AM-GM inequality we have
\medmuskip=1mu
\begin{align*}
\frac{x_{v_{i_1}}^r+x_{v_{i_2}}^r+\cdots +x_{v_{i_{r-1}}}^r}{x_{v}^r} & =
\frac{B(v,e_i)}{B(v_{i_1},e_i)}+\frac{B(v,e_i)}{B(v_{i_2},e_i)}+\cdots+
\frac{B(v,e_i)}{B(v_{i_{r-1}},e_i)}\\
& \geqslant\frac{(r-1)B(v,e_i)}{\sqrt[r-1]{B(v_{i_1},e_i)\cdots B(v_{i_{r-1}},e_i)}}\\
& =\frac{(r-1)B(v,e_i)^{\frac{r}{r-1}}}{\sqrt[r-1]{B(v,e_i)B(v_{i_1},e_i)\cdots B(v_{i_{r-1}},e_i)}}\\
& =\frac{(r-1)B(v,e_i)^{\frac{r}{r-1}}}{\alpha^{\frac{1}{r-1}}}\\
& =(r-1)\rho^{\frac{r}{r-1}}B(v,e_i)^{\frac{r}{r-1}}.
\end{align*}
It follows from H\"older inequality that
\begin{align*}
\sum_{i=1}^{d}\sum_{j=1}^{r-1}x_{v_{i_j}}^r & \geqslant
(r-1)\rho^{\frac{r}{r-1}}x_{v}^r\cdot\sum_{i=1}^dB(v,e_i)^{\frac{r}{r-1}}\\
& =(r-1)\rho^{\frac{r}{r-1}}x_{v}^r\cdot
\frac{\left[\left(\sum_{i=1}^{d}B(v,e_i)^{\frac{r}{r-1}}\right)^{\frac{r-1}{r}}\cdot
\left(\sum_{i=1}^{d}1^r\right)^{\frac{1}{r}}\right]^{\frac{r}{r-1}}}{d^{\frac{1}{r-1}}}\\
& \geqslant(r-1)\rho^{\frac{r}{r-1}}x_v^r\cdot
\frac{\left(\sum_{i=1}^dB(v,e_i)\right)^{\frac{r}{r-1}}}{d^{\frac{1}{r-1}}}\\
& = (r-1)\left(\frac{\rho^r}{d}\right)^{\frac{1}{r-1}}x_v^r.
\end{align*}
Notice that $\sum_{j=1}^nx_j^r=1$ and $H$ is linear, we obtain
\begin{align*}
1=\sum_{j=1}^nx_j^r & \geqslant x_v^r+\sum_{i=1}^d\sum_{j=1}^{r-1}x_{v_{i_j}}^r\\
& \geqslant x_v^r+(r-1)\left(\frac{\rho^r}{d}\right)^{\frac{1}{r-1}}x_v^r\\
& =\left[1+(r-1)\left(\frac{\rho^r}{d}\right)^{\frac{1}{r-1}}\right]x_v^r.
\end{align*}
Therefore we have 
\[
x_v\leqslant
\frac{1}{\sqrt[\leftroot{-2}\uproot{8}r]{1+(r-1)\left(\frac{\rho^r}{d}\right)^{\frac{1}{r-1}}}}.
\]

The proof is completed.
\end{proof}

It is known that $\rho(H)\geqslant\sqrt[r]{\Delta(H)}$ from \cite[Proposition 7.13]{Nikiforov-3},
so we get the following corollary immediately.

\begin{corollary}
Let $H$ be a connected $r$-uniform linear hypergraph on $n$ vertices with principal 
eigenvector $x=(x_1,x_2,\ldots,x_n)^{\mathrm{T}}$. Then
\[
x_i\leqslant\frac{1}{\sqrt[r]{r}}.
\]
\end{corollary}

When $r=2$, we can obtain the following classical result, which was also proved 
in \cite{Sebastian:Principal eigenvectors}. 

\begin{corollary}
[\cite{Sebastian:Principal eigenvectors}]
\label{coro:x<res2}
Suppose that $G$ is a connected graph on $n$ vertices with spectral radius $\rho$ and 
principal eigenvector $x$. For any $v\in V(G)$, if $d$ is the degree of vertex $v$, then
\[
x_v\leqslant\frac{1}{\sqrt{1+\frac{\rho^2}{d}}}.
\]
\end{corollary}

\begin{remark}
In Theorem \ref{thm:x_i upper bounds}, we give an upper bound on $x_v$ for a connected 
$r$-uniform linear hypergraph. With the same symbols as that of Theorem \ref{thm:x_i upper bounds},
Nikiforov presented a bound for general uniform hypergraph as follows \cite{Nikiforov-3}. 
For $r$-uniform hypergraph it is proved that
\begin{equation}
\label{eq:Nikiforov}
x_v\leqslant\sqrt[\leftroot{-2}\uproot{17}r]{\frac{d}{[\rho^r(r-1)!]^{\frac{1}{r-1}}}}.
\end{equation}

Notice that 
\[
1+(r-1)\left(\frac{\rho^r}{d}\right)^{\frac{1}{r-1}}>
(r-1)\left(\frac{\rho^r}{d}\right)^{\frac{1}{r-1}}\geqslant
\frac{[\rho^r(r-1)!]^{\frac{1}{r-1}}}{d}.
\]
Therefore when $H$ is a connected $r$-uniform linear hypergraph, Theorem \ref{thm:x_i upper bounds}
has a better bound than \eqref{eq:Nikiforov}. 
\end{remark}

In the following we give an upper bound on $x_{\min}$ for a connected $r$-uniform hypergraph,
which extends the result of Nikiforov \cite{Nikiforov-2} to linear hypergraphs.
\begin{theorem}
Suppose that $H$ is a connected $r$-uniform linear hypergraph on $n$ vertices with spectral radius 
$\rho$ and the principal eigenvector $x$. Let $\delta$ be the minimum vertex degree of $H$, then
\[
x_{\min}\leqslant\frac{1}{\sqrt[\leftroot{-2}\uproot{7}r]{(r-1)\left(\frac{\rho^r}{\delta}\right)^{\frac{1}{r-1}}+n-\delta(r-1)}}.
\]
\end{theorem}

\begin{proof}
Let $u\in V(H)$ be a vertex of minimum degree $d(u)=\delta$. Let $e_1$, $e_2$, $\ldots$, $e_{\delta}$ 
be all distinct edges containing $u$. Denote
\[
e_i=\{u,u_{i_1},u_{i_2},\ldots,u_{i_{r-1}}\},~i = 1,2,\ldots,\delta.
\]
Similarly to the proof in Theorem \ref{thm:x_i upper bounds}, we have
\[
\sum_{i=1}^{\delta}\sum_{j=1}^{r-1}x_{u_{i_j}}^r\geqslant
(r-1)\left(\frac{\rho^r}{\delta}\right)^{\frac{1}{r-1}}x_u^r.
\]
It follows from $\sum_{i=1}^nx_i^r=1$ that 
\[
x_{\min}^r \leqslant x_u^r\leqslant
\frac{\sum_{i=1}^{\delta}\sum_{j=1}^{r-1}x_{u_{i_j}}^r}{(r-1)\left(\frac{\rho^r}{\delta}\right)^{\frac{1}{r-1}}}
\leqslant \frac{1-[n-\delta(r-1)]x_{\min}^r}{(r-1)\left(\frac{\rho^r}{\delta}\right)^{\frac{1}{r-1}}},
\]
which implies that
\[
\left[(r-1)\left(\frac{\rho^r}{\delta}\right)^{\frac{1}{r-1}}+n-\delta(r-1)\right]
x_{\min}^r\leqslant 1.
\]

This completes the proof of this theorem.
\end{proof}

If we take $r=2$ in the above theorem, we obtain the following classical result.
\begin{corollary}
[\cite{Nikiforov-2}]
Suppose that $G$ is a connected graph on $n$ vertices with spectral radius $\rho$ and the principal 
eigenvector $x$. Let $\delta$ be the minimum vertex degree of $G$, then
\[
x_{\min}\leqslant\sqrt{\frac{\delta}{\rho^2+\delta(n-\delta)}}.
\]
\end{corollary}

\subsection{The ratio of entries in principal eigenvector and the principal ratio $\gamma$}

In this subsection we shall consider the bounds of the ratio $x_i/x_j$ for $i$, 
$j\in [n]$ as well as the principal ratio $\gamma(H)$ of a connected $r$-uniform 
hypergraph $H$.
\begin{theorem}
\label{thm:x_u/x_v-1}
Suppose that $H$ is a connected $r$-uniform hypergraph on $n$ vertices with 
spectral radius $\rho$ and the principal eigenvector $x$. For $u$, 
$v\in V(H)$, if $d(u,v)=\ell$, then
\begin{equation}
\label{eq:x_u/x_v<rho^l}
\frac{1}{\rho^{\ell}}\leqslant\frac{x_u}{x_v}\leqslant \rho^{\ell}.
\end{equation}
\end{theorem}

\begin{proof}
It is sufficient to prove the left-hand side of \eqref{eq:x_u/x_v<rho^l}. From Lemma 
\ref{lem:StrictlySubnormal}, we assume that $H$ is consistently $\alpha$-normal with 
weighted incidence matrix $B$. Let $u=u_0e_1u_1\cdots e_{\ell}u_{\ell}=v$ be a shortest 
path in $H$ from $u$ to $v$. For edge $e_i$, $i=1$, $2$, $\ldots$, $\ell$, by 
\eqref{eq:B(v,e) is constant}, we have
\begin{equation}
\label{eq:x_{u_i}B(u_i,e_i)}
x_{u_{i-1}}\cdot B(u_{i-1},e_i)^{\frac1r}=x_{u_i}\cdot B(u_i,e_i)^{\frac1r}.
\end{equation}
According to Definition \ref{defn:alpha-normal}, we see
\[
\prod_{w:w\in e_i}B(w,e_i)=\alpha~~\text{and}~~0<B(w,e_i)\leqslant 1,~
i=1,2,\ldots,\ell.
\]
Therefore we deduce that
\[
B(u_{i-1},e_i)B(u_i,e_i)\geqslant\alpha.
\]
It follows from \eqref{eq:x_{u_i}B(u_i,e_i)} that
\begin{align*}
x_{u_{i-1}}^2 & \geqslant x_{u_{i-1}}^2\cdot B(u_{i-1},e_i)^{\frac2r}\\
& =x_{u_{i-1}}x_{u_i}\cdot \left[B(u_{i-1},e_i)B(u_i,e_i)\right]^{\frac1r}\\
& \geqslant\alpha^{\frac1r}x_{u_{i-1}}x_{u_i}\\
& =\frac{x_{u_{i-1}}x_{u_i}}{\rho},
\end{align*}
which yields that
\[
\frac{x_{u_{i-1}}}{x_{u_i}}\geqslant\frac{1}{\rho},~i=1,2,\ldots,\ell.
\]
Therefore we deduce that  
\[
\frac{x_u}{x_v}=\frac{x_{u_0}}{x_{u_1}}\cdot\frac{x_{u_1}}{x_{u_2}}\cdots
\frac{x_{u_{\ell-1}}}{x_v}\geqslant\frac{1}{\rho^{\ell}}.
\]

The proof is completed. 
\end{proof}

\begin{corollary}
\label{coro:l/rho<x_u/x_v<rho}
Let $H$ be a connected $r$-uniform hypergraph with spectral radius $\rho$ and the principal 
eigenvector $x$. If $u$, $v\in V(H)$ are adjacent, then
\[
\frac{1}{\rho}\leqslant\frac{x_u}{x_v}\leqslant\rho.
\]
\end{corollary}

If spectral radius $\rho(H)\geqslant\sqrt[r]{4}$, then we have the following stronger result.

\begin{theorem}
\label{thm:x_u/x_v}
Let $H$ be a connected $r$-uniform hypergraph with spectral radius $\rho$ and 
the principal eigenvector $x$. Let $u$, $v\in V(H)$, $d(u,v)=\ell$. Then 
the following statements hold.
\begin{enumerate}
\item If $\rho>\sqrt[r]{4}$, then 
\[
\left(\frac{\sigma-\sigma^{-1}}{\sigma^{\ell+1}-\sigma^{-(\ell+1)}}\right)^{\frac2r}
\leqslant\frac{x_u}{x_v}\leqslant 
\left(\frac{\sigma^{\ell+1}-\sigma^{-(\ell+1)}}{\sigma-\sigma^{-1}}\right)^{\frac2r},
\]
where
\[
\sigma=\frac12\left(\sqrt{\rho^r}+\sqrt{\rho^r-4}\right).
\]
\item If $\rho=\sqrt[r]{4}$, then 
\[
\frac{1}{\sqrt[r]{(\ell+1)^2}}\leqslant\frac{x_u}{x_v}\leqslant\sqrt[r]{(\ell+1)^2}.
\]
\end{enumerate}
\end{theorem}

\begin{proof}
By Lemma \ref{lem:StrictlySubnormal}, we may assume that $H$ is consistently $\alpha$-normal 
with weighted incidence matrix $B$, where $\alpha=\rho^{-r}$. Let $u=u_0e_1u_1\cdots e_{\ell}u_{\ell}=v$ 
be a shortest path in $H$ from $u$ to $v$. For any $i=2$, $3$, $\ldots$, $\ell$, 
from \eqref{eq:B(v,e) is constant} we have
\[
\begin{cases}
x_{u_{i-1}}^r\cdot B(u_{i-1},e_{i-1})=
x_{u_{i-2}}^r\cdot B(u_{i-2},e_{i-1}),\\
x_{u_{i-1}}^r\cdot B(u_{i-1},e_i)=
x_{u_i}^r\cdot B(u_i,e_i).
\end{cases}
\]
Therefore we conclude that
\begin{align*}
\sqrt{x_{u_{i-2}}^r}+\sqrt{x_{u_i}^r} & =\left(\sqrt{\frac{B(u_{i-1},e_{i-1})}{B(u_{i-2},e_{i-1})}}+
\sqrt{\frac{B(u_{i-1},e_i)}{B(u_i,e_i)}}\right)\cdot \sqrt{x_{u_{i-1}}^r}\\
& =\left(\frac{B(u_{i-1},e_{i-1})}{\sqrt{B(u_{i-2},e_{i-1})B(u_{i-1},e_{i-1})}}+
\frac{B(u_{i-1},e_i)}{\sqrt{B(u_i,e_i)B(u_{i-1},e_i)}}\right)\cdot \sqrt{x_{u_{i-1}}^r}\\
& \leqslant \alpha^{-\frac12}\left[B(u_{i-1},e_{i-1})+B(u_{i-1},e_i)\right]\cdot \sqrt{x_{u_{i-1}}^r}\\
& \leqslant\alpha^{-\frac12}\cdot \sqrt{x_{u_{i-1}}^r}=\sqrt{\rho^rx_{u_{i-1}}^r}.
\end{align*}
It follows that 
\[
\sqrt{x_{u_i}^r}\leqslant\sqrt{\rho^rx_{u_{i-1}}^r}-\sqrt{x_{u_{i-2}}^r},~
i=2,3,\ldots,\ell.
\]
From Corollary \ref{coro:l/rho<x_u/x_v<rho} we have $x_{u_1}\leqslant\rho x_{u_0}=\rho x_u$.
Denote $x_{u_{-1}}=0$, we have
\begin{equation}
\label{eq:key}
\sqrt{x_{u_i}^r}\leqslant\sqrt{\rho^r}\sqrt{x_{u_{i-1}}^r}-
\sqrt{x_{u_{i-2}}^r},~i=1,2,\ldots,\ell.
\end{equation}

(1). It is sufficient to prove the left-hand side. 
We may rewrite \eqref{eq:key} as
\[
\begin{pmatrix}[1.6]
\sqrt{x_{u_i}^r}\\
\sqrt{x_{u_{i-1}}^r}
\end{pmatrix}
\leqslant
\begin{pmatrix}
\sqrt{\rho^r} & -1\\
1 & 0
\end{pmatrix}
\begin{pmatrix}[1.6]
\sqrt{x_{u_{i-1}}^r}\\
\sqrt{x_{u_{i-2}}^r}
\end{pmatrix},~
i=1,2,\ldots,\ell.
\]
If $\rho>\sqrt[r]{4}$ we have 
\begin{equation}
\label{eq:matrix form}
\begin{pmatrix}[1.6]
\sqrt{x_{u_i}^r}\\
\sqrt{x_{u_{i-1}}^r}
\end{pmatrix}
\leqslant
\begin{pmatrix}
\sqrt{\rho^r} & -1\\
1 & 0
\end{pmatrix}^i
\begin{pmatrix}[1.6]
\sqrt{x_{u_0}^r}\\
\sqrt{x_{u_{-1}}^r}
\end{pmatrix},~
i=0,1,\ldots,\ell.
\end{equation}
Observe that the fact
\[
\begin{pmatrix}
\sqrt{\rho^r} & -1\\
1 & 0
\end{pmatrix}
=
P\begin{pmatrix}
\sigma & 0\\
0 & \sigma^{-1}
\end{pmatrix}P^{-1},
\]
where
\[
P=
\begin{pmatrix}
1 & 1\\
\sigma^{-1} & \sigma
\end{pmatrix},~
\sigma=\frac12\left(\sqrt{\rho^r}+\sqrt{\rho^r-4}\right).
\]
Hence \eqref{eq:matrix form} now becomes
\[
\begin{pmatrix}[1.6]
\sqrt{x_{u_i}^r}\\
\sqrt{x_{u_{i-1}}^r}
\end{pmatrix}
\leqslant
P\begin{pmatrix}
\sigma^i & 0\\
0 & \sigma^{-i}
\end{pmatrix}P^{-1}
\begin{pmatrix}[1.6]
\sqrt{x_{u_0}^r}\\
\sqrt{x_{u_{-1}}^r}
\end{pmatrix},~
i=0,1,\ldots,\ell.
\]
Recall that $x_{u_{-1}}=0$ and $x_{u_0}=x_u$, we obtain
\begin{align*}
\begin{pmatrix}[1.6]
\sqrt{x_{u_i}^r}\\
\sqrt{x_{u_{i-1}}^r}
\end{pmatrix}
& \leqslant
\begin{pmatrix}
1 & 1\\
\sigma^{-1} & \sigma
\end{pmatrix}
\begin{pmatrix}
\sigma^i & 0\\
0 & \sigma^{-i}
\end{pmatrix}
\begin{pmatrix}
\displaystyle\frac{\sigma}{\sigma-\sigma^{-1}} & 
\displaystyle -\frac{1}{\sigma-\sigma^{-1}}\\[3mm]
\displaystyle-\frac{\sigma^{-1}}{\sigma-\sigma^{-1}} & 
\displaystyle \frac{1}{\sigma-\sigma^{-1}}
\end{pmatrix}
\begin{pmatrix}
\sqrt{x_u^r}\\
0
\end{pmatrix}\\
& =\frac{\sqrt{x_u^r}}{\sigma-\sigma^{-1}}
\begin{pmatrix}
\sigma^{i+1}-\sigma^{-(i+1)}\\
\sigma^i-\sigma^{-i}
\end{pmatrix}.
\end{align*}
Therefore we get
\[
\sqrt{x_v^r}=\sqrt{x_{u_{\ell}}^r}\leqslant 
\frac{\sigma^{\ell+1}-\sigma^{-(\ell+1)}}{\sigma-\sigma^{-1}}\sqrt{x_u^r},
\]
which implies that
\[
\frac{x_u}{x_v}\geqslant 
\left(\frac{\sigma-\sigma^{-1}}{\sigma^{\ell+1}-\sigma^{-(\ell+1)}}\right)^{\frac2r}.
\]

(2). If $\rho=\sqrt[r]{4}$, then \eqref{eq:key} can be written as
\[
\sqrt{x_{u_i}^r}-\sqrt{x_{u_{i-1}}^r}\leqslant 
\sqrt{x_{u_{i-1}}^r}-\sqrt{x_{u_{i-2}}^r},~i=1,2,\ldots,\ell.
\]
Therefore we have
\begin{equation}
\label{eq:when-rho=2}
\sqrt{x_{u_{i}}^r}-\sqrt{x_{u_{i-1}}^r}\leqslant 
\sqrt{x_{u_0}^r}-\sqrt{x_{u_{-1}}^r}=\sqrt{x_u^r},~i=0,1,\ldots,\ell.
\end{equation}
Summing over all $i$ of \eqref{eq:when-rho=2}, we obtain
\[
\sqrt{x_v^r}=\sqrt{x_{u_{\ell}}^r}-\sqrt{x_{u_{-1}}^r}\leqslant 
(\ell+1)\sqrt{x_u^r},
\]
which yields that 
\[
\frac{x_u}{x_v}\geqslant \frac{1}{\sqrt[r]{(\ell+1)^2}}.
\]

The proof is completed. 
\end{proof}

\begin{remark}
If $\rho\geqslant\sqrt[r]{4}$, Theorem \ref{thm:x_u/x_v} is much stronger than 
Theorem \ref{thm:x_u/x_v-1}. Indeed, for $\rho>\sqrt[r]{4}$ we have
\[
\rho^{\frac{r\ell}{2}}=(\sigma+\sigma^{-1})^{\ell}=
\sum_{i=0}^{\ell}\binom{\ell}{i}\sigma^{-\ell+2i}\geqslant
\sum_{i=0}^{\ell}\sigma^{-\ell+2i}=
\frac{\sigma^{\ell+1}-\sigma^{-(\ell+1)}}{\sigma-\sigma^{-1}},
\]
and $(\sqrt[r]{4})^{\ell}=(2^{\ell})^{\frac2r}\geqslant(\ell+1)^{\frac{2}{r}}=\sqrt[r]{(\ell+1)^2}$ 
for $\rho=\sqrt[r]{4}$.
\end{remark}

\begin{corollary}
Let $H$ be a connected $r$-uniform hypergraph on $n$ vertices with spectral radius $\rho>\sqrt[r]{4}$ and 
principal eigenvector $x$. Let $\ell$ be the shortest distance between a vertex having $x_{\min}$ 
and a vertex having $x_{\max}$ as their $x$ components. Then
\begin{equation}
\label{eq:gamma bound}
\gamma(H)\leqslant\left(\frac{\sigma^{\ell+1}-\sigma^{-(\ell+1)}}{\sigma-\sigma^{-1}}\right)^{\frac{2}{r}}.
\end{equation}
\end{corollary}

Taking $r=2$ in \eqref{eq:gamma bound}, we get the following result which was proved by 
Cioab$\check{\text{a}}$ and Gregory \cite{Sebastian:Principal eigenvectors}.
\begin{corollary}
[\cite{Sebastian:Principal eigenvectors}]
Let $G$ be a connected graph of order $n$ with spectral radius $\rho>2$ and principal eigenvector $x$. 
Let $\ell$ be the shortest distance from a vertex on which $x$ is maximum to a vertex on which it is 
minimum. Then
\[
\gamma(G)\leqslant\frac{\tau^{\ell+1}-\tau^{-(\ell+1)}}{\tau-\tau^{-1}},
\]
where $\tau=(\rho+\sqrt{\rho^2-4})/2$.
\end{corollary}

\section{Spectral radius of sub-hypergraph}
\label{sec4}

It is known from Perron-Frobenius theorem that if $H'$ is a proper sub-hypergraph of a 
connected $r$-uniform hypergraph $H$, then $\rho(H')<\rho(H)$. In this section, we will 
refine quantitatively the gap between $\rho(H)$ and $\rho(H')$. The main result of this 
section is inspired by that of \cite{Nikiforov-1}, and the arguments of Theorem \ref{thm:rho(H)-rho(H')}
have been used in \cite{Nikiforov-1}. For ordinary graphs, Nikiforov \cite{Nikiforov-1} 
proved the following result.
\begin{theorem}
[\cite{Nikiforov-1}]
\label{thm:Vladimir Nikiforov}
If $G'$ is a proper subgraph of a connected graph $G$ on $n$ vertices with diameter $D$, then
\[
\rho(G)-\rho(G')>\frac{1}{n\rho(G)^{2D}}.
\]
\end{theorem}

The following lemmas will be used in our proof.
\begin{lemma}
[\cite{Diameter Increase}]
\label{lem:diameter increase}
Let $G$ be a connected graph with diameter $D$, and $G'$ be any connected subgraph
obtained by deleting $t$ edges from $G$. Then $diam\,(G')\leqslant (t+1)D$.
\end{lemma}

For an $r$-uniform hypergraph $H$, we define a multiple graph $\widetilde{H}$ as follows:
$\widetilde{H}$ has vertex set $V(H)$, and vertices $u$ and $v$ are adjacent in $\widetilde{H}$ 
if and only if $\{u,v\}\subseteq e\in E(H)$.
\begin{lemma}
\label{lem:diam(H')<k(D+1)}
Let $H$ be an $r$-uniform hypergraph with diameter $D$, and $H'$ be a connected sub-hypergraph 
obtained by deleting an edge of $H$. Then $diam\,(H')\leqslant r(D+1)$.
\end{lemma}

\begin{proof}
We first prove that $d_{H}(u,v)=d_{\widetilde{H}}(u,v)$ 
for any $u$, $v\in V(H)$. Suppose that $u=u_0e_1u_1e_2\cdots u_{\ell-1}e_{\ell}u_{\ell}=v$ is a path 
in $H$ from $u$ to $v$. Let $\{u_{i-1},u_i\}=f_i$, $i=1$, $2$, $\ldots$, $\ell$, then
$u_0f_1u_1f_2\cdots f_{\ell}u_{\ell}=v$ is a walk in $\widetilde{H}$. Thus 
$d_{H}(u,v)\geqslant d_{\widetilde{H}}(u,v)$. Similarly, we have 
$d_{\widetilde{H}}(u,v)\geqslant d_{H}(u,v)$. Therefore $d_{H}(u,v)=d_{\widetilde{H}}(u,v)$. 
In particular, $diam\,(H)=diam\,(\widetilde{H})$. Let $e=\{v_{i_1},v_{i_2},\ldots,v_{i_r}\}$ 
and $H'=H-e$. Clearly,
\[
\widetilde{H'}=\widetilde{H}-\bigcup_{1\leqslant p<q\leqslant r}\{v_{i_p},v_{i_q}\},
\]
and $v_{i_1}$, $v_{i_2}$, $\ldots$, $v_{i_r}$ induces a clique $K_r$ in $\widetilde{H}$. Let 
\[
H^*=\widetilde{H}-\bigcup_{2\leqslant p<q\leqslant r}\{v_{i_p},v_{i_q}\}.
\]
Clearly, $diam\,(H^*)\leqslant D+1$, and $\widetilde{H'}=H^*-\cup_{q=2}^r\{v_{i_1},v_{i_q}\}$. It 
follows from Lemma \ref{lem:diameter increase} that $diam\,(H')=diam\,(\widetilde{H'})\leqslant r(D+1)$.
\end{proof}

The following result can be found in \cite{Nikiforov-1}.

\begin{lemma}
[\cite{Nikiforov-1}]
\label{lem:dist(w,x)+dist(v,x)}
Let $G$ be a connected graph with diameter $D$, and $G'$ be any connected subgraph
of $G$ obtained by deleting an edge $e=uv$. Then for any $w\in V(G)$ we have
\[
d_{G'}(w,u)+d_{G'}(w,v)\leqslant 2D.
\]
\end{lemma}

\begin{lemma}
\label{lem:sum diam}
Let $H$ be a connected $r$-uniform hypergraph on $n$ vertices with diameter $D$,
and $H'$ be any connected sub-hypergraph of $H$ obtained by deleting an edge $e\in E(H)$.
Then for any $w\in V(H)$ we have
\[
\sum_{v:v\in e}d_{H'}(w,v)\leqslant r(r-1)(D+1).
\]
\end{lemma}

\begin{proof}
Let $H'=H-e$, where $e=\{v_{i_1},v_{i_2},\ldots,v_{i_r}\}\in E(H)$ is an edge of $H$. 
For any $1\leqslant p, q\leqslant r$, $p\neq q$, we let 
\[
H(v_{i_p},v_{i_q})=\widetilde{H'}+\{v_{i_p},v_{i_q}\}.
\]
It is clear that 
\[
H(v_{i_p},v_{i_q})=\widetilde{H'}+\bigcup_{\substack{1\leqslant s\leqslant r\\ s\neq p}}
\{v_{i_p},v_{i_s}\}-\bigcup_{\substack{1\leqslant t\leqslant r\\ t\neq p,t\neq q}}\{v_{i_p},v_{i_t}\}.
\]
Notice that 
\[
diam\left(\widetilde{H'}+\bigcup_{1\leqslant s\leqslant r,\,s\neq p}
\{v_{i_p},v_{i_s}\}\right)\leqslant D+1.
\]
It follows from Lemma \ref{lem:diameter increase} that 
$diam\,(H(v_{i_p},v_{i_q}))\leqslant (r-1)(D+1)$. Therefore for 
any $w\in V(H)$, from Lemma \ref{lem:dist(w,x)+dist(v,x)} we have
\[
d_{\widetilde{H'}}(w,v_{i_p})+d_{\widetilde{H'}}(w,v_{i_q})
\leqslant 2\cdot diam\,(H(v_{i_p},v_{i_q}))\leqslant 2(r-1)(D+1).
\]
In particular, we have
\begin{equation}
\label{eq:k(k-1)(D+1)}
\begin{cases}
d_{\widetilde{H'}}(w,v_{i_1})+d_{\widetilde{H'}}(w,v_{i_r})\leqslant 2(r-1)(D+1),\\
d_{\widetilde{H'}}(w,v_{i_2})+d_{\widetilde{H'}}(w,v_{i_{r-1}})\leqslant 2(r-1)(D+1),\\
\hspace{19.6mm}\vdots & \\
d_{\widetilde{H'}}(w,v_{i_r})+d_{\widetilde{H'}}(w,v_{i_1})\leqslant 2(r-1)(D+1).
\end{cases}
\end{equation}
Summing the both sides of \eqref{eq:k(k-1)(D+1)} we obtain
\[
d_{\widetilde{H'}}(w,v_{i_1})+d_{\widetilde{H'}}(w,v_{i_2})+\cdots
+d_{\widetilde{H'}}(w,v_{i_r})\leqslant r(r-1)(D+1).
\]
Observe that for any $u$, $v\in V(H)$, $d_{H'}(u,v)=d_{\widetilde{H'}}(u,v)$, we have
\[
d_{H'}(w,v_{i_1})+d_{H'}(w,v_{i_2})+\cdots
+d_{H'}(w,v_{i_r})\leqslant r(r-1)(D+1).
\]

The proof is completed. 
\end{proof}

The following theorem give an estimate of spectral radii between $H$ and any
proper sub-hypergraph $H'$ of $H$. 
\begin{theorem}
\label{thm:rho(H)-rho(H')}
Suppose that $H$ is a connected $r$-uniform hypergraph on $n$ vertices with diameter
$D$ and spectral radius $\rho$. If $H'$ is a proper sub-hypergraph of $H$, then
\[
\rho(H)-\rho(H')\geqslant\min\left\{\frac{r}{n\rho^{r(r-1)(D+1)}},~
\frac{1}{n\rho^{rD}(\rho^{r-1}+r-1)}\right\}.
\]
\end{theorem}

\begin{proof}
From Perron-Frobenius theorem for nonnegative tensors, $\rho(H')\leqslant\rho(H)$ whenever 
$H'\subseteq H$. Therefore we may assume that $H'$ is a maximal proper sub-hypergraph of $H$, 
i.e., $V(H')=V(H)$ and $H'$ differs from $H$ in a single edge $e_0=\{v_1,v_2,\ldots,v_r\}$. 
We distinguish the following two cases.\par

{\bfseries Case (i).} $H'$ is connected. Let $x$ be the principal eigenvector of $H'$ and $x_w$ be the 
maximum entry of $x$. For an edge $e$ and vector $x$, we adopt the symbol $x^e:=\prod_{v\in e}x_v$ 
from \cite{LiHongHai:extremal spectral radii}. Therefore we have
\begin{align*}
\rho(H)-\rho(H') & \geqslant x^{\mathrm{T}}(\mathcal{A}(H)x)-x^{\mathrm{T}}(\mathcal{A}(H')x)\\
& =r\sum_{e\in E(H)}x^{e}-r\sum_{e\in E(H')}x^e\\
& =rx^{e_0}.
\end{align*}
From Theorem \ref{thm:x_u/x_v-1} and Lemma \ref{lem:sum diam}, we have
\begin{align*}
x^{e_0} & =\frac{x_{v_1}x_{v_2}\cdots x_{v_r}}{x_w^r}\cdot x_w^r\\
& =\frac{x_{v_1}}{x_w}\cdot\frac{x_{v_2}}{x_w}\cdots \frac{x_{v_r}}{x_w}\cdot x_w^r\\
& \geqslant\frac{x_w^r}{\rho^{d_{H'}(v_1,w)+d_{H'}(v_2,w)+\cdots+d_{H'}(v_r,w)}}\\
& \geqslant\frac{x_w^r}{\rho^{r(r-1)(D+1)}}\geqslant\frac{1}{n\rho^{r(r-1)(D+1)}}.
\end{align*}
Therefore we get
\[
\rho(H)-\rho(H')\geqslant\frac{r}{n\rho^{r(r-1)(D+1)}}.
\]

{\bfseries Case (ii).} $H'$ is disconnected. Suppose that $H'$ has $t$ connected components $H_1$, 
$H_2$, $\ldots$, $H_t$ $(t\geqslant 2)$. It is known that $\rho(H')=\max_{1\leqslant i\leqslant t}\{\rho(H_i)\}$ 
(see Lemma 9 of \cite{XiyingYuan:conjecture of Qi}). Without loss of generality, we may assume that 
$\rho(H')=\rho(H_1)$. For convenience, we denote by $|H_1|=m$, $\rho_1=\rho(H')$, and 
let $x=(x_1,x_2,\ldots,x_m)^{\mathrm{T}}$ be the principal eigenvector of $H_1$ corresponding 
to $\rho_1$. Since $H$ is connected, we know that $e_0\cap H_i\neq\emptyset$, $i=1$, $2$, 
$\ldots$, $t$. Let $\{e_0\}\cap V(H_1)=\{v_1,v_2,\ldots,v_s\}$, $1\leqslant s\leqslant r-1$. 
In the light of Theorem \ref{thm:x_u/x_v-1} and $diam\,(H_1)\leqslant D$ we have
\begin{equation}
\label{eq:rho/x0}
x_{v_j}\geqslant\frac{x_{\max}}{\rho_1^D},~j=1,2,\ldots,s,
\end{equation}
where $x_{\max}=\max\{x_1,x_2,\ldots,x_m\}$.
Let $x_0=\min\{x_{v_1},x_{v_2},\ldots,x_{v_s}\}$. Denote
\[
y=(y_1,y_2,\ldots,y_m,y_{m+1},\ldots,y_{m+r-s})^{\mathrm{T}}=
\Big(x_1,x_2,\ldots,x_m,\underbrace{\frac{x_0}{\rho_1},\ldots,\frac{x_0}{\rho_1}}_{r-s}\Big)^{\mathrm{T}}.
\]
It follows that
\[
||y||^r_r=1+(r-s)\frac{x_0^r}{\rho_1^{r}}.
\]
Therefore we have
\begin{align*}
\rho(H) & \geqslant \rho(H_1+e_0)\geqslant \frac{y^{\mathrm{T}}(\mathcal{A}(H_1+e_0)y)}{||y||^r_r}\\
& =\frac{\rho_1^r}{\rho_1^r+(r-s)x_0^r}\left(r\sum_{e'\in E(H_1)}y^{e'}+ry^{e_0}\right)\\
& =\frac{\rho_1^r}{\rho_1^r+(r-s)x_0^r}\left(x^{\mathrm{T}}(\mathcal{A}(H_1)x)+
\frac{rx_{v_1}x_{v_2}\cdots x_{v_s}x_0^{r-s}}{\rho_1^{r-s}}\right)\\
& \geqslant\frac{\rho_1^r}{\rho_1^r+(r-1)x_0^r}\left(\rho_1+\frac{rx_0^r}{\rho_1^{r-1}}\right)\\
& =\rho_1+\frac{\rho_1 x_0^r}{\rho_1^r+(r-1)x_0^r}.
\end{align*}
Notice that $x_0^r\leqslant 1\leqslant\rho_1$. Hence we obtain that
\begin{align*}
\label{eq:rho(H)-rho}
\rho(H)-\rho_1\geqslant\frac{\rho_1 x_0^r}{\rho_1^r+(r-1)x_0^r}
\geqslant\frac{\rho_1 x_0^r}{\rho_1^r+(r-1)\rho_1}
=\frac{x_0^r}{\rho_1^{r-1}+r-1}.
\end{align*}
From \eqref{eq:rho/x0} we have
\begin{align*}
\rho(H)-\rho_1 & \geqslant\frac{x_{\max}^r}{\rho_1^{rD}}\cdot\frac{1}{\rho_1^{r-1}+r-1}\\
& \geqslant\frac{1}{m\rho_1^{rD}}\cdot\frac{1}{\rho_1^{r-1}+r-1}\\
& \geqslant\frac{1}{n\rho^{rD}(\rho^{r-1}+r-1)}.
\end{align*}

The proof is completed. 
\end{proof}

\end{document}